\documentclass{article}

\usepackage[utf8]{inputenc}
\usepackage{amsmath, amssymb, amsthm}
\usepackage{mathtools}
\usepackage{authblk}
\usepackage[style=alphabetic]{biblatex}
\usepackage{fullpage}
\usepackage{hyperref}
\usepackage{cleveref}

\newcommand*{\C}{\mathbb{C}}
\newcommand*{\R}{\mathbb{R}}
\newcommand*{\Z}{\mathbb{Z}}
\newcommand*\conj[1]{\overline{#1}}
\newcommand*\norm[1]{\lvert#1\rvert}
\newcommand*\complexity[1]{\langle#1\rangle}
\newcommand*\E{\mathbb{E}}
\renewcommand*\P{\mathbb{P}}
\newcommand*\CN{\C\mathcal{N}}
\newcommand*\N{\mathcal{N}}
\newcommand*\eval[1]{\left.#1\right\rvert}
\newcommand*\opt[1]{\hat{#1}}
\newcommand*\U{\mathcal{U}}
\renewcommand*\S{\mathcal{S}}
\newcommand*\pinorm[1]{\norm{#1}_{\Pi}}

\newtheorem{theorem}{Theorem}
\newtheorem{fact}{Fact}
\newtheorem{lemma}{Lemma}
\newtheorem{corollary}{Corollary}

\newtheorem{definition}{Definition}
\newtheorem{proposition}{Proposition}
\newtheorem{claim}{Claim}

\DeclareMathOperator{\diag}{diag}
\DeclareMathOperator{\rel}{rel}
\DeclareMathOperator{\per}{per}
\DeclareMathOperator{\OPT}{OPT}
\DeclareMathOperator{\poly}{poly}
\DeclareMathOperator{\rank}{rank}
\DeclareMathOperator{\spn}{span}
\let\Re\relax
\let\Im\relax
\DeclareMathOperator{\Re}{Re}
\DeclareMathOperator{\Im}{Im}
\DeclareMathOperator{\tr}{tr}
\DeclareMathOperator{\sig}{sig}

\title{Simply Exponential Approximation of the Permanent of Positive Semidefinite Matrices}

\author[1]{Nima Anari}
\author[2]{Leonid Gurvits}
\author[3]{Shayan Oveis Gharan}
\author[1]{Amin Saberi}

\affil[1]{\small Stanford University, {\sf \{anari, saberi\}@stanford.edu}}
\affil[2]{\small The City College of New York, {\sf gurvits@cs.ccny.cuny.edu}}
\affil[3]{\small University of Washington, {\sf shayan@cs.washington.edu}}

\date{}

\bibliography{permanent}

\begin{document}
	\maketitle
	
	\abstract{
		We design a deterministic polynomial time $c^n$ approximation algorithm for the permanent of positive semidefinite matrices where $c=e^{\gamma+1}\simeq 4.84$. We write a natural convex relaxation and show that its optimum solution gives a $c^n$ approximation of the permanent. We further show that this factor is asymptotically tight by constructing a family of positive semidefinite matrices.
	}

	\section{Introduction}
	Given a matrix $A\in \C^{n\times n}$, its permanent is defined as
	\[ \per(A)=\sum_{\sigma \in S_n} \prod_{i=1}^n A_{i,\sigma(i)}, \]
	where $S_n$ is the set of permutations on $\{1,\dots,n\}$. There is a very rich body of work on permanent of matrices and its algebraic properties, see  \cite{Bap07} for a recent survey on several  theorems and open problems in this area.

	The problem has been also studied from the point of view of computational complexity. Valiant \cite{Val79} showed that it is \#P complete to compute the permanent of $\{0,1\}$-matrices. Aaronson \cite {Aar11} gave a new proof of the \#P hardness, using the model of linear optical quantum computing. In addition, he showed that it is \#P hard to compute the sign of $\per(A)$, essentially ruling out a multiplicative approximation. Grier and Schaeffer \cite{GS16} extended Aaronson's proof and proved  \#P hardness of computing the permanent of real orthogonal matrices. They also showed by a simple polynomial interpolation argument  that it is \#P hard to compute the permanent of PSD matrices. 

	Given a general matrix $A\in \R^{n\times n}$, Gurvits \cite{Gur05} designed a randomized algorithm that in time $O(n^2/\epsilon^2)$ approximates $\per(A)$ within $\pm \norm{A}^n$ additive error, where $\norm{A}$ is the largest singular value of $A$. Chakhmakhchyan, Cerf, and  Garcia-Patron \cite{CCG16} improve on Gurvits's algorithm if the matrix $A$ is PSD and its eigenvalues satisfy a certain smoothness property. 

	If all entries of $A$ are nonnegative then $\per(A)\geq 0$ by definition. In particular, if $A\in \{0,1\}^{n\times n}$, then $\per(A)$ is equal to the number of perfect matchings of the bipartite graph associated with $A$.  Jerrum, Sinclair, and Vigoda \cite{JSV04} in a breakthrough obtained a fully polynomial time randomized approximation scheme (FPRAS) for the permanent of matrices with nonnegative entries. In other words, they designed a randomized algorithm that for any given $\epsilon>0$, outputs a $1+\epsilon$ multiplicative approximation of the permanent, in time polynomial in $n$ and $1/\epsilon$. 

	The focus of this paper is on the permanent of PSD matrices, which   has received significant attention in the last decade because of its close connection to quantum optics.  In particular, permanent of PSD matrices describe output probabilities of a boson sampling experiment in which the input is a tensor product of thermal states. They form the generating function of the quantum linear optical distribution \cite{GS16}.


	It turns out that the permanent is a monotone function with respect to the Loewner order on the cone of PSD matrices and therefore the permanent of every PSD matrix is nonnegative (see \cref{lem:monotone,cor:perpos}). This fact is a priori not obvious considering that a PSD matrix can have negative entries. Since the permanent is nonnegative, unlike general matrices, there is no difficulty in computing the sign. So, it may be possible to design a polynomial time approximation scheme for the permanent of PSD matrices. This question has been posted as an open problem in several sources \cite{Aar15,GS16}. Our main result can be seen as a first step along this line.

	To this date, not much is known about multiplicative approximation of the permanent of PSD matrices. To the best of our knowledge, the only previous result is the work of Marcus \cite{Mar63} which shows that the product of the diagonal entries of a PSD matrix gives an  $n!$ approximation of the permanent. For any PSD matrix $A\in \R^{n\times n}$,
	\[ \prod_{i=1}^n A_{i,i} \leq \per(A)\leq n! \prod_{i=1}^n A_{i,i}.\]
	This approximation can be slightly improved using Lieb's permanent inequality \cite{Lieb02}. Using this inequality one can show that $\per(A)$ can be approximated to within a factor of $n!/m!^{n/m}$ in time $2^{O(m+\log n)}$.

	In this paper we design a $c^n$ multiplicative approximation algorithm for computing the permanent of PSD matrices, where $c>0$ is a universal constant. Prior to our paper, no efficient algorithm (deterministic, randomized, or quantum) was known for simply exponential approximation of the permanent of general positive semidefinite matrices.
	\begin{theorem}
		\label{thm:intro-main}
		There is a deterministic polynomial time algorithm that for any given PSD matrix $A$ returns a number $\rel(A)$ such that
		\[ \rel(A) \geq \per(A) \geq c^{-n}\rel(A) \]
		where $c=e^{\gamma+1}$ and $\gamma$ is Euler's constant.
	\end{theorem}

	Our result uses a semidefinite relaxation. Because of the aformenetioned monotonicity of the permanent with respect to the positive semidefinite order, a natural way to upper bound the permanent of a hermitian PSD matrix $A\in \C^{n\times n}$ is to find another matrix $D\succeq A$ whose permanent is easy to compute, and to use $\per(D)$ as the upper bound. For example if $D\succeq A$ is a diagonal matrix, then
	\[ \per(D)=D_{11}D_{22}\dots D_{nn} \]
	gives an easy-to-compute upper bound on $\per(A)$. This motivates the following natural relaxation for the permanent of PSD matrices.
	\begin{definition}
		\label{def:rel}
		For an $n\times n$ hermitian PSD matrix $A$ define
		\begin{equation}
			\label{eq:rel}
			\rel(A):=\inf \{\per(D): D\text{ is diagonal and }D\succeq A\}.
		\end{equation} 
	\end{definition}

	Our main result is to prove that $\rel(A)$ also lower bounds $\per(A)$ up to a multiplicative factor. Additionally, we show that $\rel(A)$ can be efficiently computed using convex programming, thus giving a polynomial-time approximation algorithm for $\per(A)$.

	\section{Preliminaries}
	
	We denote the set $\{1,\dots, n\}$ by $[n]$. We use $S_n$ to denote the set of permutations on $[n]$.
	
	\subsection{Linear Algebra}
	We identify vectors $v\in \C^n$ with $n\times 1$ matrices. For a matrix $A \in \C^{n\times m}$ we let $A^\dag\in \C^{m\times n}$ denote its conjugate transpose; in other words $(A^\dag)_{ij}=\conj{A_{ji}}$. A matrix $A\in \C^{n\times n}$ is called hermitian iff $A=A^\dag$. A hermitian matrix $A$ is called positive semidefinite (PSD) iff $v^\dag A v\geq 0$ for all $v\in \C^n$. We let $\succeq$ denote the usual Loewner order on hermitian matrices, i.e., $A\succeq B$ iff $A-B$ is PSD. For a vector $v\in \C^n$, we let $\diag(v)\in \C^{n\times n}$ denote the diagonal matrix with coordinates of $v$ as its main diagonal, i.e.,
	\[
		\diag(v):= \begin{bmatrix}
			v_1&0&\hdots&0\\
			0&v_2&\hdots&0\\
			\vdots&\vdots&\ddots&\vdots\\
			0&0&\hdots&v_n
		\end{bmatrix}.
	\]
	For matrices $A\in \C^{n\times m}$ and $B\in \C^{p\times q}$ we let $A\otimes B$ denote the Kronecker product, i.e., the following block matrix:
	\[
		A\otimes B:= \begin{bmatrix}
			A_{11}B&\hdots& A_{1m}B\\
			\vdots&\ddots&\vdots\\
			A_{n1}B&\hdots&A_{nm}B
		\end{bmatrix}.
	\]
	For a matrix $A$ and $n\geq 0$, we define $A^{\otimes n}$ as $\overbrace{A\otimes A\otimes\dots\otimes A}^n$. The Kronecker product respects the Loewner order on hermitian PSD matrices:
	\begin{fact}
		\label{fact:tensor-loewner}
		If $A\succeq B\succeq 0$ and $C\succeq D\succeq 0$, then $A\otimes C\succeq B\otimes D\succeq 0$.
	\end{fact}
	
	\subsection{Standard Complex Normal Distribution}
	We say that a complex-valued random variable $g=\Re(g)+i\Im(g)$ is distributed according to a standard complex normal, which we denote by $g\sim \CN(0, 1)$, iff $(\Re(g), \Im(g))\sim \N(0, \frac{1}{2}I)$. The probability density function (over $\C\simeq \R^2$) for this distribution is given by
	\[ \frac{1}{\pi}e^{-(\Re(g)^2+\Im(g)^2)}=\frac{1}{\pi}{e^{-\norm{g}^2}}. \]
	\begin{fact}
		\label{fact:normal-moments}
		If $g\sim \CN(0, 1)$, then for integers $n, m\geq 0$ we have
		\[ \E[g^{n}\conj{g}^m]=\begin{cases}
			0&\text{if }n\neq m,\\
			n!&\text{if }n=m.
		\end{cases}\]
	\end{fact}
	\begin{proof}
		The distribution of $g$ is circularly symmetric, i.e. for $u\in \C$ with $\norm{u}=1$, we have $ug\sim\CN(0, 1)$. This means that
		\[ \E[g^n\conj{g}^m]=\E[(ug)^n(\conj{ug})^m]=u^{n-m}\E[g^n\conj{g}^m]. \]
		Therefore, unless $n-m=0$, we have $\E[g^n\conj{g}^m]=0$. When $m=n$, we have $g^n\conj{g}^m=\norm{g}^{2n}$. If we let $r=\norm{g}\in \R_{\geq 0}$, then the probability density function of $r$ is given by $2\pi r \frac{1}{\pi}e^{-r^2}=2re^{-r^2}$. Therefore we have
		\begin{align*}
			\E[\norm{g}^{2n}]&=\int_{0}^\infty r^{2n}\cdot 2re^{-r^2}dr=\eval{-r^{2n}\cdot e^{-r^2}}_0^\infty+\int_0^\infty 2nr^{2n-1}\cdot e^{-r^2}dr\\
			&=n\int_0^\infty r^{2n-2}\cdot 2re^{-r^2}dr = n\cdot \E[\norm{g}^{2n-2}],
		\end{align*}
		where we used integration by parts. We can finally derive
		\[ \E[\norm{g}^{2n}]=n\cdot \E[\norm{g}^{2n-2}]=n(n-1)\cdot \E[\norm{g}^{2n-4}]=\dots=n!\cdot \E[\norm{g}^0]=n!.\]
	\end{proof}
	\begin{fact}
		\label{fact:normal-log}
		If $g\sim \CN(0, 1)$, then
		\[ \E[\ln(\norm{g}^2)]=-\gamma, \]
		where $\gamma$ is Euler's constant.
	\end{fact}
	\begin{proof}
		Note that $\norm{g}^2=\Re(g)^2+\Im(g)^2=\frac{1}{2}(2\Re(g)^2+2\Im(g)^2)$. Since $(\Re(g), \Im(g))\sim \N(0, \frac{1}{2}I)$, the random variable $2\Re(g)^2+2\Im(g)^2$ is distributed according to a $\chi^2$-distribution with $2$ degrees of freedom, which is identical to a $\Gamma(1, 2)$ distribution \cite{Cha93}. Therefore we have
		\[ \E[\ln(2\norm{g}^2)]=\psi(1)+\ln(2), \]
		where $\psi$ is the digamma function \cite{Cha93}. This implies that $\E[\ln(\norm{g}^2)]=\psi(1)$, and the latter is equal to $-\gamma$ \cite{AS64}.
	\end{proof}

	We say that a random vector $v\in \C^n$ is distributed according to a standard complex normal, which we denote by $v\sim \CN(0, I)$, iff $v_1,\dots,v_n$ are independent standard complex normals.
	\begin{fact}
		\label{fact:linear-functional}
		If $v\sim \CN(0, I)$, and $u\in \C^n$ is a unit vector, i.e., $\norm{u}^2=u^\dag u=1$, then $u^\dag v\sim \CN(0, 1)$.
	\end{fact}
	\begin{proof}
		Note that $(\Re(u^\dag v), \Im(u^\dag v))$ are linear combinations of the real and imaginary parts of $v$; as such, this $2$-dimensional vector is distributed according to $\N(\mu, \Sigma)$ for some $\mu\in \R^2$ and $\Sigma\in \R^{2\times 2}$.
		
		The distribution of $u^\dag v$ is circularly symmetric; i.e., if $\phi\in \C$ is such that $\norm{\phi}=1$, then $\phi u^\dag v$ is distributed the same way as $u^\dag v$. This is true because $\phi u^\dag v=u^\dag (\phi v)$, and $\phi v$ has the same distribution as $v$. Being circularly symmetric implies that $\mu=0$ and $\Sigma=cI$ for some constant $c$. On the other hand, we have
		\[ 2c=\E[\norm{u^\dag v}^2]=\E[u^\dag vv^\dag u]=u^\dag \E[vv^\dag]u=u^\dag I u=\norm{u}^2 = 1. \]
		Therefore $(\Re(u^\dag v), \Im(u^\dag v))\sim \N(0, \frac{1}{2}I)$ or in other words, $u^\dag v\sim \CN(0, 1)$.
	\end{proof}
	
	\subsection{Permanent and Loewner Order}
	For a matrix $A\in \C^{n\times n}$, its permanent is defined as
	\[ \per(A):=\sum_{\sigma\in S_n}\prod_{i=1}^n A_{i,\sigma(i)}. \]
	
	Permanent is a monotone function on the space of PSD matrices w.r.t. the Loewner order. For completeness we sketch the proof given in \cite{Bap07} here.
	\begin{lemma}
		\label{lem:pertensor}
		For any  matrix $M\in \C^{n\times n}$, there is a vector $1_{S_n}\in \C^{n^n}$
		such that 
		\[ \per(M):= \frac{1}{n!} 1_{S_n}^\dag M^{\otimes n}1_{S_n}. \]
	\end{lemma}
	\begin{proof}
		The vector $1_{S_n}\in \C^{n^n}$ is constructed in the following way: Index each of the $n^n$ coordinates by $\sigma\in [n]^n$ in the usual way (so that the indices respect the Kronecker product); we can think of $\sigma$ as a function from $[n]$ to $[n]$. Then let the $\sigma$-th coordinate of $1_{S_n}$ be $1$ iff $\sigma$ is a permutation on $[n]$, and let it be $0$ otherwise. Then, for a matrix $M$ we have
		\[ 1_{S_n}^\dag M^{\otimes n}1_{S_n}=\sum_{\sigma\in S_n}\sum_{\sigma'\in S_n}\prod_{i=1}^n M_{\sigma(i),\sigma'(i)}=\sum_{\sigma\in S_n}\per(M)=n!\cdot \per(M). \]
	\end{proof}
	
	\begin{corollary}
		\label{lem:monotone}
		If $A, B\in \C^{n\times n}$ are hermitian and $A\succeq B\succeq 0$, then
		\[ \per(A)\geq \per(B). \]
	\end{corollary}
	\begin{proof}
		The statement of the lemma follows, because $A\succeq B\succeq 0$ implies that $A^{\otimes n}\succeq B^{\otimes n}\succeq 0$ by \cref{fact:tensor-loewner}. So, by \cref{lem:pertensor}, 
		$$ \per(A) = \frac{1}{n!} 1_{S_n}^\dag A^{\otimes n}1_{S_n} \geq \frac{1}{n!} 1_{S_n}^\dag B^{\otimes n}1_{S_n} = \per(B)$$
		as desired.
	\end{proof}

	\begin{corollary}
		\label{cor:perpos}
		For any hermitian PSD matrix $A\in \C^{n\times n}$, $\per(A)	\geq 0$.
	\end{corollary}
	\begin{proof}
		This follows from \cref{lem:monotone} by setting $B=0$.
	\end{proof}
	
	There is another way to show nonnegativity of the permanent over the PSD cone with the help of the complex normal distribution. For a vector $v\in \C^n$ define
	\[ \pinorm{v} := \sqrt{\prod_{i=1}^n\norm{v_i}^2}\geq 0. \]
	Then with the help of $\pinorm{\cdot}$ we can express the permanent of a PSD matrix as an expectation of a nonnegative value.
	\begin{lemma}
		\label{lem:cauchy}
		Let $U\in \C^{d\times n}$ be arbitrary and let $x\in \C^d$ be a random vector distributed according to the standard complex normal $\CN(0, I)$. Then
		\[ \per(U^\dag U)=\E_{x\sim \CN(0, I)}[\pinorm{U^\dag x}^2]. \]
	\end{lemma}
	\Cref{lem:cauchy} is a sepcial case of the relationship between the so-called $G$-norm and the quantum permanent shown in \cite{Gur03}. In particular if the rows of $U$ are $u_1^\dag,\dots,u_d^\dag$, then
	\[ \pinorm{U^\dag x}^2 = \norm{\det(\sum_{i=1}^d x_i\diag(v_i))}^2, \]
	and therefore $\E_x[\pinorm{U^\dag x}^2]$ is the same as the $G$-norm of the polynomial $\det(\sum_{i=1}^d x_i\diag(u_i))$. In \cite{Gur03} this is shown to be equal to the quantum permanent of the linear operator with Choi form given by the matrices $\diag(u_1),\dots,\diag(u_d)$. It can be further shown that in this special case, the quantum permanent reduces to $\per(U^\dag U)$. For exact definitions and further details see \cite{Gur03}.
	
	For the sake of completeness, we give a self-contained proof of \cref{lem:cauchy} below.
	\begin{proof}[Proof of \cref{lem:cauchy}]
		We will use the fact that the expression $\pinorm{U^\dag x}^2$ is a polynomial in $x_1,\dots, x_d$ and $\conj{x_1},\dots,\conj{x_d}$; therefore we can evaluate its expectation with the help of \cref{fact:normal-moments}. We have
		\[ \pinorm{U^\dag x}^2=\norm{\prod_{i=1}^n \sum_{j=1}^d \conj{U_{ji}}x_j}^2 \]
		If we define
		\[ p(x):= \prod_{i=1}^n \sum_{j=1}^d \conj{U_{ji}}x_j, \]
		then $\pinorm{U^\dag x}^2 = p(x)\conj{p(x)}$. Note that $p(x)$ is a polynomial in terms of $x_1,\dots, x_d$. We can expand $p(x)$ as follows:
		\[ p(x)=\sum_{\sigma:[n]\to[d]}\prod_{i=1}^n \conj{U_{\sigma(i), i}}x_{\sigma(i)}, \]
		where the sum is taken over all $n^d$ functions $\sigma:[n]\to[d]$. For a function $\sigma:[n]\to[d]$, let $\sig(\sigma)$ be $(k_1,\dots,k_d)\in \Z^d$ where $k_j$ is the number of $i\in[n]$ such that $\sigma(i)=j$. Then we can alternatively write
		\[ p(x)=\sum_{\substack{k_1+\dots+k_d=n\\k_1,\dots,k_d\geq 0}}\left(x_1^{k_1}\dots x_d^{k_d}\sum_{\substack{\sigma:[n]\to[d]\\ \sig(\sigma)=(k_1,\dots,k_d)}}\prod_{i=1}^n \conj{U_{\sigma(i), i}}\right). \]
		For $(k_1,\dots,k_d)\neq (k_1',\dots,k_d')$, by \cref{fact:normal-moments} we have $\E_x[x_1^{k_1}\dots x_d^{k_d}\conj{x_1^{k_1'}\dots x_d^{k_d'}}]=0$. Therefore we can write
		\[ \E_x[p(x)\conj{p(x)}]=\sum_{\substack{k_1+\dots+k_d=n\\k_1,\dots,k_d\geq 0}}\left(k_1!\dots k_d! \adjustlimits\sum_{\substack{\sigma:[n]\to[d]\\ \sig(\sigma)=(k_1,\dots,k_d)}}\sum_{\substack{\sigma':[n]\to[d]\\ \sig(\sigma')=(k_1,\dots,k_d)}}\prod_{i=1}^n \conj{U_{\sigma(i),i}}U_{\sigma'(i),i}\right), \]
		where we used that $\E[x_1^{k_1}\dots x_d^{k_d}\conj{x_1^{k_1}\dots x_d^{k_d}}]=k_1!\dots k_d!$ by \cref{fact:normal-moments}. Note that when $\sig(\sigma)=\sig(\sigma')$, there is a permutation $\pi\in S_n$ such that $\sigma'=\sigma\circ \pi$. In fact if $\sig(\sigma)=\sig(\sigma')=(k_1,\dots,k_d)$, then the number of $\pi\in S_n$ for which $\sigma'=\sigma\circ \pi$ is exactly equal to $k_1!\dots k_d!$. Therefore we can rewrite the above sum as
		\begin{align*} \E_x[p(x)\conj{p(x)}]&=\sum_{\sigma:[n]\to[d]}\sum_{\pi\in S_n} \prod_{i=1}^n \conj{U_{\sigma(i), i}}U_{\sigma(\pi(i)), i}=\sum_{\pi\in S_n}\sum_{\sigma[n]\to[d]}\prod_{i=1}^n (U^\dag)_{i, \sigma(i)}U_{\sigma(i), \pi^{-1}(i)}\\
			&=\sum_{\pi\in S_n}\prod_{i=1}^n\sum_{j=1}^d (U^\dag)_{i, j}U_{j, \pi^{-1}(i)}=\sum_{\pi\in S_n}\prod_{i=1}^n (U^\dag U)_{i, \pi^{-1}(i)}=\per(U^\dag U).
		\end{align*}

	\end{proof}

	\section{Approximation of Permanent on the PSD Cone}

	
	
	
	
	In this section we prove \cref{thm:intro-main}.
	Recall the definition of $\rel(A)$ from \cref{def:rel}. Our first step is to prove that for every $n\times n$ hermitian PSD matrix $A\succeq 0$:
		\begin{equation}
			\label{thm:main}
			c^n\per(A)\geq \rel(A),
		\end{equation}
		where $c=e^{\gamma+1}$.
		
	
	In order to prove \cref{thm:main}, we also introduce a lower bound on $\per(A)$. We find a vector $v\in \C^n$ such that $A\succeq vv^\dag$.  By \cref{lem:monotone},  $\per(A)\geq \per(vv^\dag)$. So in order to prove \cref{thm:main} it suffices to prove: 
	\begin{theorem}
		\label{thm:rank1}
		For a hermitian PSD matrix $A\in \C^{n\times n}$, there exists $v\in \C^n$ such that $A\succeq vv^\dag$ and
		\[ c^n \per(vv^\dag)\geq \rel(A), \]
		where $c=e^{\gamma+1}$.
	\end{theorem}
	
	Note that the above shows that for every hermitian PSD matrix $A\in \C^{n\times n}$, there exists a diagonal matrix $D$ and a rank 1 matrix $vv^\dag$ such that
	\[ D\succeq A\succeq vv^\dag, \]
	and $\per(D)\leq c^n \per(vv^\dag)$ for $c=e^{\gamma+1}$. Thus $\per(A)$ is sandwiched between $\per(D)$ and $\per(vv^\dag)$, two quantities that differ by at most a simply exponential factor. 
	
	It is also worth noting that there is no additional loss in approximating $\per(A)$ by the permanent of a rank one matrix.  In \cref{sec:tight}, we will show that the constant $e^{\gamma+1}$ is not only asymptotically tight in \cref{thm:rank1}, but also in \cref{thm:main}. 
	
	Another interesting corollary of \cref{thm:rank1} is that  that instead of $\rel(A)$ we can use  $\per(vv^\dag)$  as an approximation of $\per(A)$, with the same $e^{n(\gamma+1)}$ approximation factor:
	\begin{equation} \label{eq:rank1} \sup\{\per(vv^\dag):v\in \C^n\text{ and }A\succeq vv^\dag\}. \end{equation}
	Moreover, $\per(vv^\dag)$ is easily computable.
	\begin{fact}
		\label{fact:rank1}
		For a vector $v\in \C^n$, we have $\per(vv^\dag)=n!\cdot \prod_{i=1}^n\norm{v_i}^2$.
	\end{fact}
	\begin{proof}
		For any permutation $\sigma\in S_n$ we have
		\[ \prod_{i=1}^n (vv^\dag)_{i, \sigma(i)}=\prod_{i=1}^n v_i\conj{v_{\sigma(i)}}=\prod_{i=1}^n v_i \cdot \prod_{i=1}^n \conj{v_i}=\prod_{i=1}^n\norm{v_i}^2. \]
		Since $\per(vv^\dag)$ is the sum of the above quantity for all $\sigma\in S_n$, we get that $\per(vv^\dag)=n!\cdot \prod_{i=1}^n\norm{v_i}^2$.
	\end{proof}
	
	Even though $\per(vv^\dag)$ has a closed form, we do not have an efficient way of computing the $\sup$ in \cref{eq:rank1}, whereas, as we show in \cref{sec:computing}, $\rel(A)$ can be computed efficiently. 
	
	The next section is dedicated to proving \cref{thm:rank1}. To finish up the proof of \cref{thm:intro-main} we need to design an algorithm to compute $\rel(A)$ for a given PSD matrix $A$.
	\begin{theorem}
		\label{thm:comp}
		There is an algorithm that outputs an $e^{n(\gamma+1)}$-approximation of $\per(A)$ for any hermitian PSD $A\in \C^{n\times n}$ in time $\poly(n+\complexity{A})$, where $\complexity{A}$ represents the bit complexity of $A$.
	\end{theorem}
We will prove the above theorem in \cref{sec:computing}.
\Cref{thm:rank1,thm:comp} together complete the proof of \cref{thm:intro-main}. In \cref{sec:tight} we show that the constant $c=e^{\gamma+1}$ in \cref{thm:main} is asymptotically tight.

	\subsection{Proof of the Main Result}
	\label{sec:proof}
	
	In order to prove \cref{thm:rank1}, we use a seemingly unrelated quantity about distributions on unit vectors $\{u\in \C^d: \norm{u}^2=u^\dag u=1\}$. Let us define this quantity below.
	\begin{definition}
		For a discrete distribution $\U$ supported on the sphere $\{u\in \C^d:\norm{u}^2=u^\dag u=1\}$, define
		\[ f(\U):=\sup_{x\in \spn(\U)}\left\{ \frac{e^{\E_{u\sim \U}[\ln(\norm{u^\dag x}^2)]}}{\E_{u\sim \U}[\norm{u^\dag x}^2]} \right\}, \]
		where $\spn(\U)$ is the span of the support of $\U$, i.e., the set of vectors for which the denominator is nonzero.
	\end{definition}
	We will prove \cref{thm:rank1} by showing that there exists $v\in \C^n$ such that $A\succeq vv^\dag$ and
	\[ \per(vv^\dag)\geq \frac{n!}{n^n}f(\U)^{n}\cdot \rel(A), \]
	where $\U$ is an appropriately constructed distribution on unit vectors. The expression $n!/n^n$ is lower bounded by $e^{-n}$. Thus if we show that $f(\U)\geq e^{-\gamma}$, the above inequality would imply the multiplicative factor of $e^{n(\gamma+1)}$ desired in \cref{thm:rank1}.
	
	To gain some intuition about $f(\U)$, note that by Jensen's inequality, applied to the concave function $\ln$, it is easy to see that $f(\U)\leq 1$:
	\[ \frac{e^{\E_{u\sim \U}[\ln(\norm{u^\dag x}^2)]}}{\E_{u\sim \U}[\norm{u^\dag x}^2]}\leq \frac{e^{\ln(\E_{u\sim \U}[\norm{u^\dag x}^2])}}{\E_{u\sim \U}[\norm{u^\dag x}^2]}=1. \]
	On the other hand, we will show that for all $\U$, $f(\U)\geq e^{-\gamma}$.
	\begin{proposition}
		\label{prop:fu}
		For all discrete distributions $\U$ supported on the sphere $\{u\in \C^d:\norm{u}^2=u^\dag u=1\}$,
		\[ f(\U)\geq e^{-\gamma}. \]
	\end{proposition}
	This universal lower bound is independent of the dimension $d$ or the size of the support of $\U$. We defer the proof of \cref{prop:fu} to the end of this section.
	
	Let us now prove \cref{thm:rank1}, assuming correctness of \cref{prop:fu}.
	\begin{proof}[Proof of \cref{thm:rank1}]
		Let us break down the proof into a series of claims, and then prove them one by one.
		\begin{claim}
			\label{claim:one}
			The infimum in \cref{eq:rel} is achieved by some diagonal matrix $\opt{D}=\opt{D}(A)$.	In other words there exists a diagonal matrix $\opt{D}\succeq A$ such that $\per(\opt{D})=\rel(A)$.
		\end{claim}
		\begin{claim}
			\label{claim:two}
			We may assume without loss of generality that $\opt{D}=I$.
		\end{claim}
		\begin{claim}
			\label{claim:three}
			The first-order optimality condition of $\opt{D}$ implies that there exists a correlation matrix $B\in \C^{n\times n}$, i.e., a hermitian PSD matrix with $1$s on its main diagonal, such that $AB=B$.
		\end{claim}
		We may use the Cholesky decomposition to write $B=U^\dag U$ where $U\in \C^{d\times n}$ for $d=\rank(B)$.
		\begin{claim}
			\label{claim:four}
			For any $x\in \C^d$ the vector $v=U^\dag x/\norm{U^\dag x}$ satisfies
			 \[ A\succeq vv^\dag. \]
		\end{claim}
		Naturally we may want to choose $x$ so as to maximize $\per(vv^\dag)$.
		\begin{claim}
			\label{claim:five}
			We have
			\[ \sup_{x\in \C^d}\{\per(vv^\dag)\}=\frac{n!}{n^n}f(\U)^n, \]
			where $\U$ is the uniform distribution on the columns of $U$.
		\end{claim}
		And now the statement of \cref{thm:rank1} follows, because $\rel(A)=\per(\opt{D})=1$ when $\opt{D}=I$; we have found $v\in \C^n$ such that $A\succeq vv^\dag$ and
		\[ e^{n(\gamma+1)}\per(vv^\dag)\geq \frac{n^n}{n!}f(\U)^{-n}\per(vv^\dag)\geq 1=\rel(A). \]
		Let us now prove the claims one by one.
		\begin{proof}[Proof of \cref{claim:one}] 
			We divide the proof into two cases. First assume that $A_{ii}>0$ for all $i\in [n]$.
			Let $\lambda\geq 0$ be larger than the maximum eigenvalue of $A$. Then $\lambda I\succeq A$. This proves that $ \rel(A)\leq \lambda^n$.
			Note that $D\succeq A$ implies $D_{ii}\geq A_{ii}$ for all $i\in [n]$.  If any entry $D_{ii}$ of $D$ satisfies
			\[ D_{ii}> \frac{\lambda^nA_{ii}}{\prod_{j=1}^n A_{jj}}, \]
			then
			\[ \per(D)> \frac{\lambda^n A_{ii}}{\prod_{j=1}^n A_{jj}}\prod_{j\neq i} A_{jj}=\lambda^n. \]
			This effectively eliminates such a $D$ as a candidate for the $\inf$ in \cref{eq:rel}. Therefore we may take $\inf$ of $\per(D)$ over the set of all diagonal matrices $D$ which in addition to $D\succeq A$ satisfy
			\[ D_{ii}\leq \frac{\lambda^nA_{ii}}{\prod_{j=1}^n A_{jj}} \]
			for all $i\in [n]$. This is a compact set, and $\per(D)$ is a continuous function. Therefore the $\inf$ is achieved by some matrix $\opt{D}$.
			
			For the second case, assume that $A_{ii}=0$ for some $i$. Then since $A$ is PSD, the $i$-th row and the $i$-th column of $A$ are both zero.
			Let $\lambda$ be larger than the largest eigenvalue of $A$. Define $\opt{D}$ by $\opt{D}_{ii}=0$ and $\opt{D}_{jj}=\lambda$ for $j\neq i$. It is easy to see that $\opt{D}\succeq A$ and $\per(\opt{D})=0$. Therefore $\rel(A)=0$ and it is achieved at $\opt{D}$.
		\end{proof}
		\begin{proof}[Proof of \cref{claim:two}]
			First note that without loss of generality we may assume $\opt{D}(A)\succ 0$, since otherwise $\rel(A)=0$ and the conclusion of \cref{thm:rank1} is trivial.
			
			Now let $\lambda\in\R_{>0}^n$ be an arbitrary positive vector and define $T_\lambda: \C^{n\times n}\to \C^{n\times n}$ by
			\[ T_\lambda(M)=\diag(\lambda)M\diag(\lambda). \]
			Note that $T_\lambda$ respects the Loewner order and maps diagonal matrices to diagonal matrices. It is one-to-one and surjective on the space of diagonal matrices. The matrix $T_\lambda(M)$ is obtained from $M$ by multiplying column $i$ by $\lambda_i$ for $i\in [n]$ and then row $i$ by $\lambda_i$ for $i\in [n]$. Therefore
			\[ \per(T_\lambda(M))=\lambda_1^2\dots \lambda_n^2\per(M). \]
			This implies that
			\begin{align*} \lambda_1^2\dots \lambda_n^2\rel(A)&=\inf\{\lambda_1^2\dots\lambda_n^2\per(D):D\text{ is diagonal and }D\succeq A\}\\&=\inf\{\per(T_\lambda(D)):T_\lambda(D)\text{ is diagonal and }T_\lambda(D)\succeq T_\lambda(A)\}=\rel(T_\lambda(A)). \end{align*}
			It is also easy to see that the above also implies $\opt{D}(T_\lambda(A))=T_\lambda(\opt{D}(A))$. In particular if $\lambda$ is set so that $\lambda_i=1/\sqrt{\opt{D}_{ii}}$, then $\opt{D}(T_\lambda(A))=I$. So we can replace $A$ by $T_\lambda(A)$ and continue the proof of \cref{thm:rank1} to find $v\in \C^n$ satisfying
			\[ T_\lambda(A)\succeq vv^\dag, \]
			and $c^n\per(vv^\dag)\geq \rel(T_\lambda(A))=1$ with $c=e^{\gamma+1}$. Let $w=\diag(\lambda)^{-1}v$. Then $T_\lambda(ww^\dag)=vv^\dag$. This implies that
			\[ A\succeq ww^\dag, \]
			and
			\[ c^n \per(ww^\dag)=\frac{1}{\lambda_1^2\dots\lambda_n^2}c^n\per(vv^\dag)\geq \frac{1}{\lambda_1^2\dots\lambda_n^2}\rel(T_\lambda(A))=\rel(A). \]
		\end{proof}
		\begin{proof}[Proof of \cref{claim:three}]
			We use the first-order optimality condition of $\per(D)$ at $D=I$. Let us change $I$ to $I+X$ where $X$ is a diagonal matrix. Then if $X$ is small enough $\per(I+X)\simeq 1+\tr(X)$. More precisely, we have
			\[ \eval{\frac{d}{dt}\per(I+tX)}_{t=0}=\eval{\frac{d}{dt}\prod_{i=1}^n(1+tX_{ii})}_{t=0}=\sum_{i=1}^nX_{ii}=\tr(X). \]
			If $I+X\succeq D$ then $I+tX\succeq D$ for all $t\in [0, 1]$. If $\tr(X)<0$, then for small enough $t$, $\per(I+tX)<\per(I)$ which contradicts the fact that $\opt{D}(A)=I$. This implies that the optimal solution of the following SDP is $0$:
			\[
				\begin{array}{llr}
					\min_{X} & \tr(X)& \\
					\text{subject to} & I+X\succeq A&\\
					& X_{ij}=0&\forall i\neq j
				\end{array}
			\]
			The dual of this SDP has variables $B\succeq 0$, corresponding to the constraint $I+X\succeq A$, and $\mu_{ij}$ for $i\neq j$, corresponding to the constraint $X_{ij}=0$:
			\[
				\begin{array}{llr}
					\max_{B, \mu_{ij}} & \tr((A-I)B)\\
					\text{subject to} & B_{ij}+\mu_{ij}=0& \forall i\neq j\\
						& B_{ii}=1& \forall i\\
						& B\succeq 0
				\end{array}
			\]
			Because of strong duality, the optimum of this SDP is $0$. The optimal $B$ satisfies $B\succeq 0$ and $B_{ii}=1$ for $i\in [n]$, i.e., $B$ is a correlation matrix. We also have $\tr((I-A)B)=0$. But since $I-A\succeq 0$ and $B\succeq 0$, this implies that $(I-A)B=0$ or in other words $AB=B$.
		\end{proof}
		\begin{proof}[Proof of \cref{claim:four}]
			We have $B=U^\dag U$ with $U\in \C^{d\times n}$ and $\rank(B)=d$. This implies that $UU^\dag\in \C^{d\times d}$ is invertible. Now we have
			\[ BU^\dag (UU^\dag)^{-1}x=U^\dag UU^\dag(UU^\dag)^{-1}x=U^\dag x. \]
			This together with $AB=B$ implies that
			\[ AU^\dag x = ABU^\dag(UU^\dag)^{-1}x=BU^\dag(UU^\dag)^{-1}x=U^\dag x. \]
			In other words, $U^\dag x$ is an eigenvector of $A$ with eigenvalue $1$. This means that $v=U^\dag x/\norm{U^\dag x}$ is also such an eigenvector. So $Av=v$ and $\norm{v}=1$.  We conclude that $A\succeq vv^\dag$.
		\end{proof}
		\begin{proof}[Proof of \cref{claim:five}]
			Let us compute $\per(vv^\dag)$. By \cref{fact:rank1} we have
			\[ \per(vv^\dag)=n!\cdot \prod_{i=1}^n \norm{v_i}^2. \]
			Let the columns of $U$ be $u_1,\dots,u_n\in \C^d$. Then $v_i=u_i^\dag x/\norm{U^\dag x}$, and note that $\norm{U^\dag x}^2=\sum_{i=1}^n \norm{u_i^\dag x}^2$. We can rewrite $\per(vv^\dag)$ as
			\[ \per(vv^\dag)=n!\cdot \frac{\prod_{i=1}^n \norm{u_i^\dag x}^2}{(\sum_{i=1}^n \norm{u_i^\dag x}^2)^n}=\frac{n!}{n^n}\cdot \left(\frac{\sqrt[n]{\prod_{i=1}^n\norm{u_i^\dag x}^2}}{\frac{1}{n}\sum_{i=1}^n \norm{u_i^\dag x}^2}\right)^n. \]
			Now if we let $\U$ be the uniform distribution on $u_1,\dots,u_n$, we can rewrite the above as
			\[ \per(vv^\dag) = \frac{n!}{n^n} \cdot \left(\frac{\exp(\E_{u\sim \U}[\ln(\norm{u^\dag x}^2)])}{\E_{u\sim \U}[\norm{u^\dag x}^2]}\right)^n\]
			Therefore
			\[ \sup_{x\in \C^d}\{\per(vv^\dag)\}=\frac{n!}{n^n}f(\U)^n. \]
		\end{proof}
		This concludes the proof of \cref{thm:rank1}.
	\end{proof}
	
	It only remains to prove \cref{prop:fu}.
	\begin{proof}[Proof of \cref{prop:fu}]
		Without loss of generality we may assume that $\spn(\U)=\C^d$; if that is not the case, we can identify $\spn(\U)$ with $\C^{d'}$  for some $d'<d$ using a unitary transformation and nothing changes.
		
		Let $x\sim \CN(0, I)$ be a $d$-dimensional standard complex normal.  Let 
		\begin{align*}
			g(x)&=\exp(\E_{u\sim \U}[\ln(\norm{u^\dag x}^2)]),\\
			h(x)&=\E_{u\sim \U}[\norm{u^\dag x}^2].
		\end{align*}
		Then our goal is to prove that $\P_x[g(x)/h(x)\geq e^{-\gamma}]>0$ or equivalently $\P_x[g(x)-e^{-\gamma}h(x)\geq 0]>0$. To this end, we will prove that $\E_x[g(x)-e^{-\gamma}h(x)]\geq 0$, and the conclusion follows.
		
		By \cref{fact:linear-functional}, for each fixed $u$ in the support of $\U$, $u^\dag x\sim \CN(0, 1)$. Therefore we have
		\[ \E_x[h(x)]=\E_x\E_u[\norm{u^\dag x}^2]=\E_u\E_x[\norm{u^\dag x}^2]=\E_u[1]=1. \]
		On the other hand by \cref{fact:normal-log} we have
		\begin{align*}
			\E_x[g(x)]&=\E_x[\exp(\E_u[\ln(\norm{u^\dag x}^2)])]\geq \exp(\E_x\E_u[\ln(\norm{u^\dag x}^2)])\\ &=\exp(\E_u\E_x[\ln(\norm{u^\dag x}^2)])=\exp(\E_u[-\gamma])=e^{-\gamma},
		\end{align*}
		where the inequality is an application of Jensen's to the convex function $\exp$. Putting these together we get that $\E_x[g(x)-e^{-\gamma}h(x)]\geq e^{-\gamma}-e^{-\gamma}=0$ as desired.
	\end{proof}
	
	\subsection{Computing the Approximation}
	\label{sec:computing}
	
	In this section we show how to approximately compute $\rel(A)$. The main result of this section will be \cref{thm:comp}.
	
	The main ingredient of the proof is transforming $\rel(D)$ to the objective of a convex program. The original optimization problem that computes $\rel(D)$ is the following:
	\[
		\begin{array}{ll}
			\min_{D}& D_{11}\dots D_{nn}\\
			\text{subject to}& D\succeq A\\
			& D\text{ is diagonal}
		\end{array}
	\]
	
	The objective is not concave, even if we apply $\ln$ to it. The trick is to change from the variables $D_{11},\dots, D_{nn}$ to $D_{11}^{-1},\dots,D_{nn}^{-1}$. If we have the Cholesky decomposition $A=V^\dag V$ for some $V\in \C^{d\times n}$, then $D\succeq A$ if and only if
	\[ I\succeq V D^{-1}V^\dag. \]
	So we can turn the optimization problem into the following by identifying $D^{-1}$ with $\diag(x)$.
	\begin{equation}
		\label{eq:cp}
		\begin{array}{llr}
			\min_{x\in \R^n}& -\ln(x_1\dots x_n)\\
			\text{subject to}&I\succeq V \diag(x)V^\dag\\
			&x_i\geq 0&\forall i
		\end{array}
	\end{equation}
	If the objective of the above program is $\OPT$, then $\rel(A)=e^{\OPT}$. Note that $-\ln(x_1\dots x_n)$ is convex over $\R_{\geq 0}^n$, so the above is a valid convex program.
	\begin{proof}[Proof of \cref{thm:comp}]
		We can detect whether $\rel(A)=0$ by checking whether any of $A$'s main diagonal entries are $0$. See the proof of \cref{claim:one}.
		
		When all of the main diagonal entries of $A$ are strictly positive, similar to the proof of \cref{claim:one}, we can determine upper and lower bounds on the optimum $x_i$. In particular if $\lambda$ is a number larger than the largest eigenvalue of $A$, for the optimum $x_i$ we have
		\[ A_{ii}^{-1}\geq x_i\geq \frac{\prod_{j=1}^n A_{jj}}{\lambda^n A_{ii}}. \]
		Thus, we can restrict the domain of the convex program in \cref{eq:cp} to a compact bounded domain. We can compute the Cholesky decomposition of $A$ and then use our favorite convex programming technique, such as the ellipsoid method, to find the optimum value of \cref{eq:cp} to within accuracy $\epsilon$ in time $\poly(n+\complexity{A}+\log(1/\epsilon))$. This gives us a $1+\epsilon$ approximation of $\rel(A)$ which by \cref{thm:main} is a $(1+\epsilon)c^n$ approximation of $\per(A)$ for $c=e^{\gamma+1}$.
		
		As a final remark, we note that the approximation factor $e^{n(\gamma+1)}$ in \cref{thm:main} can in fact be slightly strengthened to
		\[ \frac{n^n}{n!}e^{n\gamma}, \]
		if one carefully reviews the proof. The term $n^n/n!$ is at most $e^n$, but the difference allows us to absorb $1+\epsilon$ into the approximation factor for an appropriately chosen $\epsilon$. This allows us to state an $\epsilon$-free result: We can find an $e^{n(\gamma+1)}$ approximation to $\per(A)$ in time $\poly(n+\complexity{A})$.
	\end{proof}
	\section{Asymptotically Tight Examples}
	\label{sec:tight}
	In this section we show that the constant $c=e^{\gamma+1}$ cannot be replaced by anything smaller in \cref{thm:main}. In other words we will construct $n\times n$ hermitian PSD matrices $A$ such that
	\[ \sqrt[n]{\frac{\rel(A)}{\per(A)}}\to e^{\gamma+1}. \]
	
	The construction will begin with a distribution $\U$ that is uniform over $n$ unit vectors $u_1,\dots,u_n\in \C^d$. We will later show how we can construct $\U$ so that $f(\U)$ is arbitrarily close to $e^{-\gamma}$.
	\begin{lemma}
		\label{lem:uconst}
		For any $\epsilon>0$ there exists a distribution $\U$ that is uniform over $n$ unit vectors $u_1,\dots,u_n\in \C^d$ for some $n$ and $d$ that satisfies
		\[ f(\U)\leq e^{-\gamma}+\epsilon. \]
	\end{lemma}
	We postpone the proof of \cref{lem:uconst} to the end of this section. For now we use it to show the following. The following proposition together with \cref{lem:uconst} show that $e^{\gamma+1}$ cannot be improved in \cref{thm:main}.
	\begin{proposition}
		\label{prop:tight}
		Given a distribution $\U$ that is uniform over a finite number of unit vectors $u_1,\dots, u_n$, we can construct a sequence of matrices $A_1,A_2,\dots$ of sizes $n_1\times n_1,n_2\times n_2,\dots$ such that
		\[ \sqrt[n_k]{\frac{\rel(A_k)}{\per(A_k)}}\to ef(\U)^{-1}. \]
	\end{proposition}
	
	\begin{proof}
	Our goal is to construct a PSD matrix $A$ and relate $\rel(A)/\per(A)$ to $f(\U)$. We will assume without loss of generality that $\spn\{u_1,\dots,u_n\}=\C^d$; otherwise, we use a unitary transformation to map $u_1,\dots, u_n$ onto a lower dimensional space and $f(\U)$ would not change.
	
	Consider the matrix $U\in \C^{d\times n}$ whose columns are $u_1,\dots,u_n$. Note that $\rank(U)=d$ and $U^\dag U\succeq 0$ has $1$s on the main diaognal. In other words $U^\dag U$ is a correlation matrix of rank $d$. Since $\rank(U)=d$, the matrix $UU^\dag$ is invertible and we can define
	\[ V:=(UU^\dag)^{-1/2}U, \]
	and
	\[ A:=V^\dag V=U^\dag(UU^\dag)^{-1}U.\]
	We will study $\rel(A)$ and $\per(A)$ and relate them to $f(\U)$.
	
	As observed in the proof of \cref{claim:three}, correlation matrices can be used as optimality certificates for $\rel$, albeit in that context first order optimality was just a necessary condition. We now make a formal claim by certifying that $\rel(A)=1$ using $U^\dag U$ as the certificate.
	\begin{claim}
		\label{claim:relbound}
		If $A$ is constructed as above, then
		\[ \rel(A)=\rel(V^\dag V)=1. \]
	\end{claim}
	\begin{proof}
		We clearly have $I\succeq U^\dag(UU^\dag)^{-1}U=V^\dag V$. This implies that $\rel(A)\leq 1$. Now consider a diagonal matrix $D\succeq A=V^\dag V$. We need to show that $\per(D)\geq 1$. Without loss of generality, by adding a small multiple of $I$ if necessary, we may assume that $D\succ 0$. Now $D\succeq V^\dag V$ implies that
		\[ I\succeq V D^{-1}V^\dag, \]
		which in turn implies
		\[ UU^\dag=(UU^\dag)^{1/2}(UU^\dag)^{1/2}\succeq (UU^\dag)^{1/2}VD^{-1}V^\dag(UU^\dag)^{1/2}=UD^{-1}U^\dag. \]
		By taking the trace we get
		\[ \tr(U^\dag U)=\tr(UU^\dag)\geq \tr(UD^{-1}U^\dag)=\tr(D^{-1}U^\dag U). \]
		Since $U^\dag U$ has $1$s on the diagonal and $D$ is diagonal the above becomes
		\[ n\geq \sum_{i=1}^n D_{ii}^{-1}. \]
		By using the AM-GM inequality we get
		\[ (D_{11}^{-1}\dots D_{nn}^{-1})^{1/n}\leq \frac{\sum_{i=1}^n D_{ii}^{-1}}{n}\leq 1. \]
		This means that $\per(D)=D_{11}\dots D_{nn}\geq 1$.
	\end{proof}
	
	Next we study $\per(A)$. This is where the term $f(\U)$ appears.
	\begin{claim}
		\label{claim:perbound}
		If $A$ is constructed as above, then
		\[ \per(A)\leq \frac{n!}{n^n}\binom{n+d-1}{d-1} \cdot f(\U)^n. \]
	\end{claim}
	Before proving \cref{claim:perbound}, let us show why it suffices to finish the proof of \cref{prop:tight}. By \cref{claim:relbound} and \cref{claim:perbound} we have
	\[
		\sqrt[n]{\frac{\rel(A)}{\per(A)}}\geq \sqrt[n]{\frac{n^n}{n!}}\cdot \sqrt[n]{\binom{n+d-1}{d-1}^{-1}}\cdot f(\U)^{-1}.
	\]
	This is not quite the same as $ef(\U)^{-1}$ yet. However we have one degree of freedom we have not used. Initially we assumed $\U$ was a uniform distribution over $n$ unit vectors. But we might have as well assumed that it is a uniform distribution over $nk$ unit vectors for any integer $k$, by simply repeating the vectors in the support of $\U$. Therefore we may make $n$ as large as we would like without changing $d$ or $f(\U)$. As $n\to\infty$, by Stirling's formula we have
	\[ \sqrt[n]{\frac{n^n}{n!}}\to e, \]
	and by a simple bound for large enough $n$
	\[ \sqrt[n]{\binom{n+d-1}{d-1}^{-1}}\geq \sqrt[n]{n^{-d}}\to 1. \]
	Therefore as $n\to\infty$ we have
	\[ \sqrt[n]{\frac{\rel(A)}{\per(A)}}\to ef(\U)^{-1}. \]
	
	It only remains to prove \cref{claim:perbound}.
	\begin{proof}[Proof of \cref{claim:perbound}]
		We will use \cref{lem:cauchy} to write down $\per(A)=\per(V^\dag V)$. Let $x\in \C^d$ be distributed according to a $d$-dimensional standard complex normal $\CN(0, I)$. Then according to \cref{lem:cauchy} we have
		\[ \per(A)=\E_{x\sim \CN(0, I)}[\pinorm{V^\dag x}^2]. \]
		Our goal is to use $f(\U)$ to bound $\pinorm{V^\dag x}$. According to the definition of $f(\U)$, for the vector $y=(UU^\dag)^{-1/2}x$ we have
		\[ \frac{\sqrt[n]{\prod_{i=1}^n \norm{u_i^\dag y}^2}}{\frac{1}{n}\sum_{i=1}^n \norm{u_i^\dag y}^2}=\frac{\exp(\E_{u\sim \U}[\ln(\norm{u^\dag y}^2)]}{\E_{u\sim \U}[\norm{u^\dag y}^2]}\leq f(\U). \]
		Note that
		\[ u_i^\dag y=(U^\dag y)_i=(U^\dag (UU^\dag)^{-1/2}x)_i=(V^\dag x)_i. \]
		This means that $\prod_{i=1}^n \norm{u_i^\dag y}^2=\pinorm{V^\dag x}^2$. We also have
		\[ \sum_{i=1}^n \norm{u_i^\dag y}^2=x^\dag VV^\dag x=x^\dag (UU^\dag)^{-1/2}UU^\dag (UU^\dag)^{-1/2} x=x^\dag x=\norm{x}^2. \]
		Putting these together we get
		\[ \pinorm{V^\dag x}^2 \leq \left(\frac{f(\U)\norm{x}^2}{n}\right)^n=\left(\frac{f(\U)}{n}\right)^n\norm{x}^{2n}. \]
		Let us now compute $\E_x[\norm{x}^{2n}]$. We have
		\[ \E_x[\norm{x}^{2n}]=\E_x[\prod_{j=1}^n(\sum_{i=1}^d \norm{x_i}^2)]=\sum_{\substack{k_1,\dots, k_d\geq 0\\k_1+\dots+k_d=n}}\binom{n}{k_1,\dots,k_d}\E_x[\norm{x_1}^{2k_1}\dots\norm{x_d}^{2k_d}]. \]
		According to \cref{fact:normal-moments}, we have $\E_x[\norm{x_1}^{2k_1}\dots \norm{x_d}^{2k_d}]=k_1!\dots k_d!$. Therefore
		\[ \E_x[\norm{x}^{2n}]=\sum_{\substack{k_1,\dots, k_d\geq 0\\k_1+\dots+k_d=n}}\binom{n}{k_1,\dots,k_d}k_1!\dots k_d!=\sum_{\substack{k_1,\dots, k_d\geq 0\\k_1+\dots+k_d=n}}n!=n!\binom{n+d-1}{d-1}, \]
		where in the last equality we used the fact the number of ways to write $n$ as a sum of $d$ nonnegative integers is $\binom{n+d-1}{d-1}$. We conclude by getting
		\[ \per(A)\leq \left(\frac{f(\U)}{n}\right)^n\E_x[\norm{x}^{2n}]=\frac{n!}{n^n}\binom{n+d-1}{d-1}f(\U)^n. \]
	\end{proof}
	
	This finishes the proof of \cref{prop:tight}.
	\end{proof}
	
	Now we switch gears and construct the distribution $\U$ promised by \cref{lem:uconst}.
	\begin{proof}[Proof of \cref{lem:uconst}]
		The idea is to make $\U$ be close to the uniform distribution on the sphere $\{u\in \C^d:\norm{u}=1\}$ for some large $d$. If we were allowed to pick $\U$ to be uniform over the sphere, then intuitively all choices of $x$ in the definition of $f(\U)$ would yield the same value and we would be able to argue about this common value using the same tricks as in the proof of \cref{prop:fu}. Instead we use the uniform distribution on a large number of samples from the sphere to serve as the proxy for the uniform distribution on the sphere itself. We further need the dimension $d$ to grow, to make the uniform distribution on the sphere similar to a (scaled) normal distribution. We now make these formal.
		
		Let us fix some $d$ and let $\S$ denote the uniform distribution on the sphere $\{u\in \C^d:\norm{u}=1\}$. For any fixed distance $\epsilon$ we can cover the sphere by a finite number of balls $B(o_1,\epsilon),\dots, B(o_m, \epsilon)$ where $o_1,\dots, o_m$ are unit vectors and
		 \[ B(o, \epsilon)=\{v\in \C^d: \norm{o-v}\leq \epsilon \}. \]
		 Let $n$ be a large number and draw $n$ random points $u_1,\dots,u_n$ from $\S$. We will let $\U$ be the uniform distribution over $u_1,\dots, u_n$. We would like to argue that $f(\U)$ is with high probability close to $f(\S)$. Because the sphere was covered by the balls around $o_i$'s, for each unit vector $x$ we have $\norm{x-o_i}\leq \epsilon$ for some $i$. This implies that
		 \begin{align*} \E_{u\sim \U}[\ln(\norm{u^\dag x}^2)]&\leq \E_{u\sim \U}[\ln((\norm{u^\dag o_i}+\epsilon)^2)], \\
		  \E_{u\sim \U}[\norm{u^\dag x}^2]&\geq \E_{u\sim \U}[\max(0, \norm{u^\dag o_i}-\epsilon)^2]. \end{align*}
		 On the other hand by the law of large numbers for each $o_i$ we have with high probability as $n\to \infty$
		 \begin{align*}
		 	\E_{u\sim \U}[\ln((\norm{u^\dag o_i}+\epsilon)^2)] &\to \E_{u\sim \S}[\ln((\norm{u^\dag o_i}+\epsilon)^2)], \\
		 	\E_{u\sim \U}[\max(0,\norm{u^\dag o_i}-\epsilon)^2] &\to \E_{u\sim \S}[\max(0, \norm{u^\dag o_i}-\epsilon)^2]. \\
		 \end{align*}
		 Let us condition on the event that the LHS of the above are sufficiently close to the RHS for all $o_i$. This event happens with high probability as $n\to\infty$. Note that because of symmetry, the RHS of the above are independent of the choice of $o_i$. Under this condition we have for all unit vectors $x$
		 \[ \frac{\exp(\E_{u\sim \U}[\ln(\norm{u^\dag x}^2)])}{\E_{u\sim \U}[\norm{u^\dag x}^2]}\leq \frac{\exp(\E_{u\sim \S}[\ln((\norm{u^\dag o}+\epsilon)^2)])}{\E_{u\sim \S}[\max(0, \norm{u^\dag o}-\epsilon)^2]}+\delta, \]
		 where $o$ is any arbitrary vector and
		 $\delta\to 0$ as $n\to \infty$. The above bounds the LHS for unit vectors $x$. However note that the LHS does not change if we scale $x$ by any constant. Therefore $f(\U)$ is bounded by the RHS. As we take the limit with $\epsilon\to 0$ and $\delta\to 0$ we get $\U$ with $f(\U)$ asymptotically bounded by $f(\S)$.
		 
		 Now it only remains to show that as the dimension $d$ grows $f(\S)\to e^{-\gamma}$. Let $o$ be an arbitrary point with $\norm{o}^2=d$ such as $\sqrt{d}e_1$ where $e_1$ is the first element of the standard basis. When $u\sim \S$ is a random point on the sphere, we would like to argue that $u^\dag o$ is almost distributed like $\CN(0, 1)$. If this were the case we would have
		 \[ f(S) = \frac{\exp(\E_{u}[\ln(\norm{u^\dag o}^2)])}{\E_{u}[\norm{u^\dag o}^2]} \simeq \frac{\exp(\E_{g\sim\CN(0, 1)}[\ln(\norm{g}^2)])}{\E_{g\sim\CN(0, 1)}[\norm{g}^2]}=e^{-\gamma}, \]
		 where in the last equality we used \cref{fact:normal-log}.
		 
		 To make this approximation rigorous, let us generate the random point $u$ on the sphere by the following process: We sample a standard $d$-dimensional complex normal $v\sim \CN(0, I)$ and then we let $u=v/\norm{v}$. We have $u^\dag o=v_1\frac{d}{\norm{v}}$. Therefore
		 \[ \E_u[\ln(\norm{u^\dag o}^2)]=\E_v[\ln(\norm{v_1}^2)]+2\ln(d)-2\E_v[\ln(\norm{v}^{2})]. \]
		 The random variable $\norm{v}^2$ is distributed according to a $\frac{1}{2}$-scaled $\chi^2$-distribution with $2d$ degrees of freedom which is the same as $\Gamma(d, 1)$. We can therefore write
		 \[ \E_v[\ln(\norm{v}^2)]=\psi(d)=\ln(d-1)+o(1), \]
		 where $\psi$ is the digamma function \cite{Cha93, AS64}. We therefore have $\E_u[\ln(\norm{u^\dag o}^2)]=-\gamma+o(1)$.
		 
		 For $\E_u[\norm{u^\dag o}^2]$ we observe that
		 \[ \E_u[\norm{u^\dag o}^2]=d\cdot \E_v\left[\frac{\norm{v_1}^2}{\norm{v}^2}\right].\]
		 The random variables $\norm{v_i}^2/\norm{v}^2$ are identically distributed for different $i$. As such we have
		 \[ \E_u[\norm{u^\dag o}^2]=d\cdot \E_v\left[\frac{\norm{v_1}^2}{\norm{v}^2}\right]=\E_v\left[\frac{\norm{v_1}^2}{\norm{v}^2}\right]+\dots+\E_v\left[\frac{\norm{v_d}^2}{\norm{v}^2}\right]=\E_v\left[\frac{\norm{v}^2}{\norm{v}^2}\right]=1. \]
		 Therefore
		 \[ \frac{\exp(\E_u[\ln(\norm{u^\dag o}^2)])}{\E_u[\norm{u^\dag o}^2]}=e^{-\gamma+o(1)}. \]
		 This shows that $f(\S)\to e^{-\gamma}$ as $d\to\infty$ and concludes the proof.
	\end{proof}

%
	
	\printbibliography
\end{document}